\documentclass[a4paper,oneside,10pt]{amsart}
\usepackage{bm}
\usepackage[pdftex]{graphicx,xcolor} 
\usepackage{ascmac}
\usepackage{amsmath}
\usepackage{amscd}
\usepackage{amssymb}
\usepackage{amsfonts}
\usepackage{amsthm}
\usepackage[all]{xy}
\usepackage{mathtools}
\usepackage{enumerate}
\usepackage[pdftex]{hyperref}
\usepackage{tikz}
\usepackage{hyperref}

\newtheorem{prop}{Proposition}[section]

\newtheorem{thm}{Theorem}[section]

\newtheorem{obs}{Observation}

\newcommand{\ES}{\operatorname{ES}}
\newcommand{\ESl}{\ES_L}

\title{Happy Ending or many concurrent lines}
\author{Koki Furukawa}


\begin{document}

\begin{abstract}
The Erd\"{o}s-Szekeres (convex polygon) problem \cite{ES} (also known as Happy Ending problem) is a classical problem about a convex polygon. Erd\"{o}s and Szekeres proved that, for all $n \geq 3$, a sufficiently large set of points in general position in the plane always contains $n$ points in convex position. 
B\'{a}r\'{a}ny, Rold\'{a}n and G.T\'{o}th created a line version of this problem \cite{ESFL}.
They proved that for every $n \geq 2$, a sufficiently large family of lines which form a simple arrangement of lines in the plane always contains $n$ lines which determine the bound of an $n$-cell. 
In this paper, we extend their results to arbitrary family of lines where no two lines are parallel.
For each $n \geq 2$, $l \geq 3$, let $\ESl (l,n)$ be the minimum $N$ such that every family of $N$-lines in the plane contains either $l$ concurrent lines or $n$ lines in convex position. 
We give an upper and lower bound of $\ESl (l,n)$. 
\end{abstract}

\maketitle
\section{Introduction}
Throughout the paper, we fix $x$ and $y$ coordinates of the plane.
A finite set of points is in \textit{general position} if no three of the points are collinear.
A finite set of points is in \textit{convex position} if every point are on the boundary of the convex hull of the set.
Erd\"{o}s and Szekeres proved that for every $n \geq 3$, there is a minimum number $\ES (n)$ such that any set of $\ES (n)$-points in general position in the plane always contains an $n$-subset of points in convex position \cite{ES}.
Finding the exact value of $\ES (n)$ has long been studied.
Erd\"{o}s and Szekeres gave the upper bound of $\ES (n)$ in \cite{ES} which is $\ES (n) \leq \binom{2n-4}{n-2} + 1$.
And they constructed $2^{n-2}$ points in the plane containing no $n$ points in convex position which shows 
$\ES (n) \geq 2^{n-2} + 1$.
This lower bound is believed to be tight and is called the \textit{Erd\"{o}s-Szekeres Conjecture}. It is known to be correct for $n \leq 6$ and it still remains open for $n \geq 7$.
Since their upper bound $\binom{2n-4}{n-2} + 1$ is roughly $O\left( \frac{4^n}{\sqrt{n}} \right)$, there was a huge gap between their upper and lower bounds. 
After the several decades, Suk \cite{SUK} gave a tremendously better upper bound of $\ES (n)$ which is $\ES (n) \leq 2^{n + O(n^{\frac{2}{3}} \log n)}$.
Currently, the best upper bound is due to Holmsen, Mojarrad, Pach and Tardos \cite{HOLMSEN}. They showed that $\ES (n) \leq 2^{n + O(\sqrt{n \log n})}$ by optimizing Suk's argument.

Several variants of this problem have been studied.
B\'{a}r\'{a}ny, Rold\'{a}n and G.T\'{o}th considered the line version of this result \cite{ESFL}.
A finite family of lines in the plane is in \textit{general position} (simple line arrangement) if no three are concurrent, no two are parallel and no lines are vertical.
Given a family of $n$-lines $L$ in the plane and let $n \geq k \geq 2$.
$L$ defines a $k$-\textit{cell} $P$ if $P$ is a connected component of the complement of the union of these lines, and the boundary of $P$ contains a segment of positive length from exactly $k$ of the lines. A family of $n$ lines is in \textit{convex position} if it defines an $n$-cell.
They proved that for every $n \geq 2$, there is a minimum number $\ESl (n)$ such that any family of lines of size $\ESl (n)$ in general position in the plane always contains $n$ lines in convex position. 
They also gave the upper and lower bounds for $\ESl (n)$ which are
$$
2 {\binom{n-4}{\lceil \frac{n}{2} \rceil - 1}}^2 < \ESl (n) \leq \binom{2n-4}{n-2} \quad \mbox{for} \,\, n \geq 5.
$$
Notably, the order of $\ES (n)$ is $2^{n + O(\sqrt{n \log n})}$ while the order of $\ESl(n)$ is $4^n/n^\alpha$ for some $1/2 < \alpha < 1$.

In this paper, we extended their results to almost arbitrary family of lines in the plane as follows.
We say that a finite family of lines is in \textit{nearly general position} if no two are parallel and all lines are non-vertical.
We extend the problem to a family of lines in nearly general position.
For each $n \geq 2$, $l \geq 3$, let $\ESl (l,n)$ be the minimum $N$ such that every family of $N$-lines in the plane in nearly general position contains either $l$ concurrent lines or $n$ lines in convex position. 
Therefore, $\ESl (3,n) = \ESl (n)$.
Hence, finding the value of $\ESl (l, n)$ is the extension of \cite{ESFL}.
From another point of view, 
it can be regarded as a dual version of \cite{BLBC}.
In this paper, we give an upper bound and lower bound of $\ESl (l, n)$ as the following thoerems.

\begin{thm} \label{thm1.1}
There exists $c > 1$ such that, for each $l,n \geq 3$,
$$
\ESl (l,n) \leq c(n+l-1) \cdot \binom{2n-4}{n-2}.
$$
\end{thm}

For $l \geq 3$, we obtain the following.
\begin{thm} \label{thm1.2}
For an integer $k \geq 2$,
if $n = 2k + 2$,
$$
\ESl (l,n) > (l-1) {\binom{2k-2}{k-1}}^2 - (l-3) \binom{2k-4}{k-2}  \binom{2k-2}{k-1}.
$$
If $n = 2k + 1$,
$$
\ESl (l,n) > (l-1) \bigg\{ \binom{2k-2}{k-1} + 1 \bigg\} \cdot \binom{2k-3}{k-1} - (l-3) \binom{2k-4}{k-2} \binom{2k-2}{k-1}.
$$
\end{thm}

Here the order of the upper and lower bounds are roughly $l \frac{4^n}{n}  \lesssim \ESl (l,n) \lesssim c(n+l-1)\frac{4^n}{\sqrt{n}}$.
 From now on, we assume throughout that our families of lines are always in nearly general position.

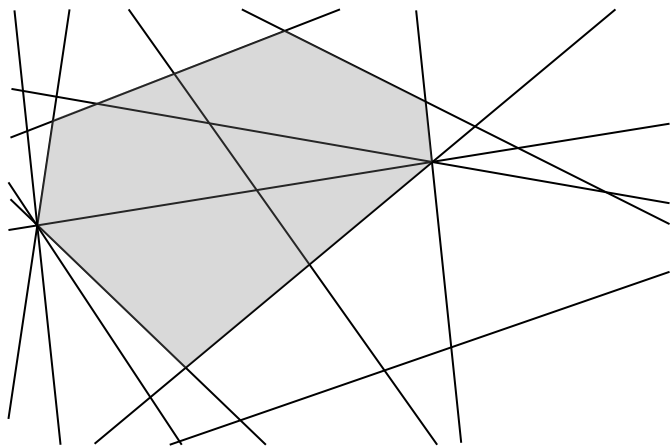
\begin{figure}[h]
\tikzset{every picture/.style={line width=0.75pt}} 

\begin{tikzpicture}[x=0.75pt,y=0.75pt,yscale=-0.5,xscale=0.5]

\draw    (160,438) -- (659,264) ;
\draw    (1,191) -- (256,438) ;
\draw    (85,437) -- (605,0) ;
\draw    (-1,222) -- (659,115) ;
\draw    (-1,412) -- (60,0) ;
\draw    (1,129) -- (330,0) ;
\draw    (2,80) -- (659,195) ;
\draw    (232,0) -- (659,216) ;
\draw    (406,1) -- (451,436) ;
\draw    (119,0) -- (427,438) ;
\draw  [draw opacity=0][fill={rgb, 255:red, 128; green, 128; blue, 128 }  ,fill opacity=0.3 ] (43,111) -- (27,217) -- (175,361) -- (422,153) -- (415,93) -- (275,21) -- cycle ;
\draw    (5,1) -- (51,438) ;
\draw    (-1,174) -- (172,438) ;
\end{tikzpicture}
\caption{A $6$-cell.}
\end{figure}

\section{Upper bound}

We say that a family of $n$-lines $\mathcal{F}$ forms an $n$-\textit{cup} ($n$-\textit{cap}) if $\mathcal{F}$ is in convex position and corresponding $n$-cell $C$ has the property that its intersection with vertical lines is a half-line bounded from below (above).
Note that any single non-vertical line forms a $1$-cup and a $1$-cap.
An $n$-cell $C$ defined by a family of lines $L = \{ l_1, \ldots, l_n \}$ (suppose they are ordered according to slope) is said to be an $n$-\textit{cell unbounded to the right (left)} if $C$ is unbounded and for two end lines $l_1$ and $l_n$, $l_1 \cap cl(C)$ and $l_n \cap cl(C)$ are both half-lines which is unbounded to the right (left) where $cl(C)$ denotes the closure of $C$.

\begin{figure}[h]
\tikzset{every picture/.style={line width=0.75pt}} 
\begin{tikzpicture}[x=0.75pt,y=0.75pt,yscale=-0.65,xscale=0.65]

\draw    (2,43) -- (79,178) ;
\draw    (10,105) -- (112,162) ;
\draw    (187,99) -- (58,173) ;
\draw    (187,43) -- (100,166) ;
\draw    (8,13) -- (37,149) ;
\draw    (408.43,146.55) -- (313.31,23.64) ;
\draw    (391.84,86.27) -- (282.87,44.1) ;
\draw    (228,131) -- (334.8,25.66) ;
\draw    (225.25,172.42) -- (294.19,38.46) ;
\draw    (407.69,174.09) -- (359.95,43.48) ;
\draw    (659,11) -- (473,97) ;
\draw    (523,27) -- (477,181) ;
\draw    (574,20) -- (472,137) ;
\draw    (661.5,146.5) -- (474.5,170.5) ;
\draw  [draw opacity=0][fill={rgb, 255:red, 0; green, 0; blue, 0 }  ,fill opacity=0.16 ] (8,13) -- (187,43) -- (122.5,136) -- (95,152) -- (50,129) -- (21,76) -- cycle ;
\draw  [draw opacity=0][fill={rgb, 255:red, 0; green, 0; blue, 0 }  ,fill opacity=0.16 ] (325,60) -- (349,70) -- (387,121) -- (407.69,174.09) -- (225.25,172.42) -- (265,94) -- (292,68) -- cycle ;
\draw  [draw opacity=0][fill={rgb, 255:red, 0; green, 0; blue, 0 }  ,fill opacity=0.16 ] (531,70) -- (659,11) -- (661.5,146.5) -- (481,170) -- (498,109) -- cycle ;
\draw    (34,166) -- (185,132) ;
\draw    (1,26) -- (62,178) ;
\draw    (414,167) -- (347,48) ;
\draw    (231,82) -- (382,48) ;
\draw    (1,18) -- (174,168) ;
\end{tikzpicture}
\caption{A cup, a cap and a $4$-cell unbounded to the right.}
\end{figure}
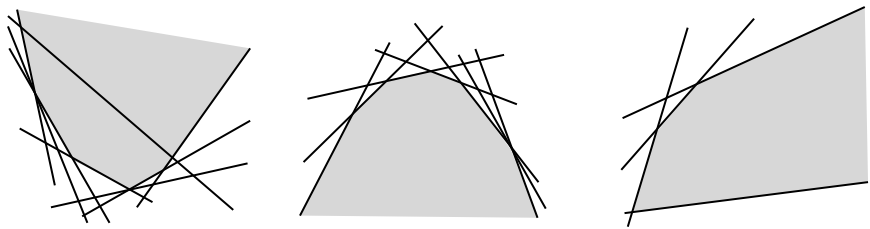

Let $f_L (l,p,q)$ be the minimum $N$ such that every family of $N$-lines in the plane contains either $l$ concurrent lines, a $p$-cup or a $q$-cap.\\
In \cite{BLBC}, they defined the number
$f (l,p,q)$ as the minimum $N$ such that every $N$-point set in the plane (suppose each points have distinct $x$-coordinates) contains either $l$ collinear members, a $p$-cup or a $q$-cap.
Here, we say that an $n$-point set $X$ in the plane forms an $n$-\textit{cup} ($n$-\textit{cap}) if the points of $X$ lie on the graph of a convex (concave) function.
They proved the following.

\begin{thm} $($\cite{BLBC}$)$.
There is an absolute constant $c > 1$ such that, for $p,q,l \geq 3$,
$$
f (l,p,q) \leq c( \min\{p-1, q-1\} + l) \cdot \binom{p+q-4}{q-2}.
$$
\end{thm}

\,\\
\textbf{Remark.}
According to \cite{BECK}, there is an absolute constant $\varepsilon > 0$ such that every $t$-elements point set in the plane contains either $\varepsilon t$ collinear points or determines at least $\varepsilon \binom{t}{2}$ distinct lines.
Theorem $2.1$ requires absolute constant $c > 1$ to be at least $\frac{10}{\varepsilon}$.
\\

To prove Theorem $1.1$, we consider the dual transformation $D$ between the set of non-vertical lines in the plane
and the set of points in the plane. That is,
$$
D : \{ y = mx + c \} \mapsto (m,c).
$$
This transformation satisfies the following properties.
\begin{obs}
$l_1, \ldots, l_n$ form an $n$-cup $($$n$-cap$)$ if and only if $D(l_1), \ldots, D(l_n)$ form an $n$-cup $($$n$-cap$)$.
\end{obs}

\begin{obs}
For each $l \geq 3$,
points $p_1, \ldots, p_l$ are collinear if and only if lines $D(p_1), \ldots, D(p_l)$ are concurrent.
\end{obs}

\begin{proof}[Proof of Theorem \ref{thm1.1}]
From these observations, $N$-point set in the plane containing either $l$ collinear members, a $p$-cup or a $q$-cap has a one-to-one correspondence to the family of $N$-lines in the plane containing either $l$ concurrent lines, a $p$-cup or a $q$-cap.
Hence, we obtain
$
f_L (l,p,q) = f (l,p,q).
$
Then it follows from Theorem $2.1$ that for some absolute constant $c>1$,
$$
\ESl (l, n) \leq f_L (l, n,n) \leq c(n + l - 1) \cdot \binom{2n-4}{n-2}.
$$
\end{proof}

\section{Lower bound}
Recall that our families of lines are always in nearly general position.
Before proving the lower bound, we calculate the exact value of $\ESl (l, n)$ for fixed $l$ and $n \leq 4$.
Since an non-vertical line defines no $2$-cell and any family of $2$-lines defines $2$-cells, $\ESl (l, 2) = 2$.
For each $2 \leq i \leq l - 1$, a family of $i$-lines $L_i$ in which they intersect at a single point always defines some $2$-cells but no $3$-cell.
Since $\ESl (3) = 3$ \cite{ESFL}, a family of $l$-lines formed by adding one non-vertical line to $L_{l-1}$ always forms a $3$-cell or $l$ concurrent lines.
Hence, $\ESl (l, 3) = (l-1) + 1 = l$.
$\ESl (4) = 4$ shows that a family of the union of $L_{l-1}$ and two non-vertical lines contains $4$-cells or $l$ concurrent lines.
Therefore, $\ESl (l, 4) = l + 1$.

For a non-vertical line $a$, $\varepsilon > 0$, $l \geq 3$ and a family of lines $\mathcal{F}$, 
we can find a proper affine transformation so called \textit{unbounded-cell-preserving affine transformation} such that the image of $\mathcal{F}$ is the family of lines $\mathcal{F} (a, \varepsilon)$ satisfying the following properties \cite{ESFL}:

\begin{enumerate}[(i)]
\item the slopes of all lines in $\mathcal{F} (a, \varepsilon)$ are within $\varepsilon$ of the slope $a$;
\item all intersections of the lines of $\mathcal{F} (a, \varepsilon)$ are below the $x$-axis;
\item the distance between any two intersections of the lines of $\mathcal{F} (a, \varepsilon)$ is at most $\varepsilon$;
\end{enumerate}

\begin{figure}[htbp]
\tikzset{every picture/.style={line width=0.75pt}} 

\begin{tikzpicture}[x=0.552pt,y=0.5pt,yscale=-1,xscale=1]

\draw    (71.06,421) -- (293.11,155) ;
\draw    (137.67,1) -- (156.33,418) ;
\draw    (0.89,265) -- (293.11,177) ;
\draw    (9.77,1) -- (294,307) ;
\draw    (0,397) -- (174.98,1) ;
\draw    (294,133) -- (0,289) ;
\draw    (224.72,0) -- (107.47,419) ;
\draw    (360,423) -- (618,59) ;
\draw    (629,3) -- (371.29,420.47) ;
\draw    (366.47,353.86) -- (614,125) ;
\draw    (616,1) -- (395.59,420.47) ;
\draw    (363.24,404.45) -- (626,25) ;
\draw    (617,95) -- (364.04,377.47) ;
\draw    (623,42) -- (365.21,421.31) ;
\draw [color={rgb, 255:red, 155; green, 155; blue, 155 }  ,draw opacity=1 ][fill={rgb, 255:red, 155; green, 155; blue, 155 }  ,fill opacity=1 ]   (368,119) -- (632,118.01) ;
\draw [shift={(635,118)}, rotate = 179.79] [fill={rgb, 255:red, 155; green, 155; blue, 155 }  ,fill opacity=1 ][line width=0.08]  [draw opacity=0] (12.5,-6.01) -- (0,0) -- (12.5,6.01) -- cycle    ;
\draw [color={rgb, 255:red, 208; green, 2; blue, 27 }  ,draw opacity=1 ]   (623,62) -- (363,390) ;
\draw   (408.33,195.98) -- (435.81,214.2) -- (437.19,212.12) -- (452.76,228.42) -- (431.68,220.43) -- (433.06,218.35) -- (405.57,200.13) -- cycle ;
\draw   (426.33,174.98) -- (453.81,193.2) -- (455.19,191.12) -- (470.76,207.42) -- (449.68,199.43) -- (451.06,197.35) -- (423.57,179.13) -- cycle ;
\draw   (393.33,219.98) -- (420.81,238.2) -- (422.19,236.12) -- (437.76,252.42) -- (416.68,244.43) -- (418.06,242.35) -- (390.57,224.13) -- cycle ;
\draw   (545.37,292.21) -- (518.55,273.03) -- (517.1,275.06) -- (502.12,258.21) -- (522.9,266.95) -- (521.45,268.98) -- (548.26,288.16) -- cycle ;
\draw   (526.63,312.56) -- (499.81,293.37) -- (498.36,295.4) -- (483.38,278.56) -- (504.16,287.3) -- (502.71,289.32) -- (529.53,308.51) -- cycle ;
\draw   (561.21,268.76) -- (534.39,249.58) -- (532.94,251.6) -- (517.96,234.76) -- (538.74,243.5) -- (537.29,245.52) -- (564.11,264.71) -- cycle ;

\draw (136,449.4) node [anchor=north west][inner sep=0.75pt]    {$\mathcal{F}$};
\draw (494,446.4) node [anchor=north west][inner sep=0.75pt]    {$\mathcal{F}( a,\ \varepsilon )$};
\draw (371,311.4) node [anchor=north west][inner sep=0.75pt]  [color={rgb, 255:red, 208; green, 2; blue, 27 }  ,opacity=1 ]  {$a$};
\draw (638,110.4) node [anchor=north west][inner sep=0.75pt]  [color={rgb, 255:red, 74; green, 74; blue, 74 }  ,opacity=1 ]  {$x$};

\end{tikzpicture}
\caption{}
\end{figure}

Intuitively, a family of lines $\mathcal{F} (a, \varepsilon)$ can be regarded as a contraction of $\mathcal{F}$ toward a line $a$ while preserving how they intersect.
Moreover, this affine transformation preserves cups, caps, cells unbounded to the right (left) and concurrency.
Hence, if a family of lines $\mathcal{F}$ has no $(p+1)$-cup, $(q+1)$-cap, $l$ concurrent lines, and no $4$-cell unbounded to the right, then $\mathcal{F} (a, \varepsilon)$ also contains no $(p+1)$-cup, $(q+1)$-cap, $l$ concurrent lines, and no $4$-cell unbounded to the right. See Figure 3.

\begin{prop}
For each $p,q \geq 2, l \geq 3,$ there is a family $\mathcal{F}_{p,q}^l$ consisting of at least $\frac{l-1}{2} \binom{p+q-2}{q-1} - \frac{l-3}{2} \binom{p+q-4}{q-2}$ lines that contains neither $l$ concurrent lines, $(p+1)$-cup, $(q+1)$-cap nor $4$-cell unbounded to the right.
\end{prop}

\begin{proof}
For $p \geq 2, l \geq 3$, we construct $\mathcal{F}_{p,2}^l$ as follows: 
Consider the positive slope lines $l_1, \ldots, l_{\lfloor \frac{p}{2} \rfloor}$ forming a $\lfloor \frac{p}{2} \rfloor$-cup. For each $i \in \{ 1, \ldots, \lfloor \frac{p}{2} \rfloor \}$, let $s_i \subset l_i$ be the segment of the boundary of the $\lfloor \frac{p}{2} \rfloor$-cell. Choose points $a_i$ from the interior of each $s_i$.
For each $i \in \{ 1, \ldots, \lfloor \frac{p}{2} \rfloor \}$, replace $l_i$ with a family of $(l-1)$ concurrent lines $L_i$ such that they all intersect at $a_i$ and the slopes of all the lines in $L_i$ are within $\varepsilon > 0$ of the slope $l_i$ where $\varepsilon > 0$ to be sufficiently small so that all the lines in $L_i$ pass below $a_{i+1}$.
 If $p$ is odd, we can add another line that has greater slope than any of the lines in $L_{\lfloor \frac{p}{2} \rfloor}$ and passes below $a_{\lfloor \frac{p}{2} \rfloor}$.
See Figure 4.
Therefore, define $\mathcal{F}_{p,2}^l$ as
$$
\mathcal{F}_{p,2}^l \coloneqq
\begin{cases}
 L_1 \cup \cdots \cup L_{\lfloor \frac{p}{2} \rfloor} & (\mbox{if} \,\, p \,\, \mbox{is even})\\
 L_1 \cup \cdots \cup L_{\lfloor \frac{p}{2} \rfloor} \cup  \{ \mbox{one positive slope line}\} & (\mbox{if} \,\, p \,\, \mbox{is odd})
\end{cases}.
$$

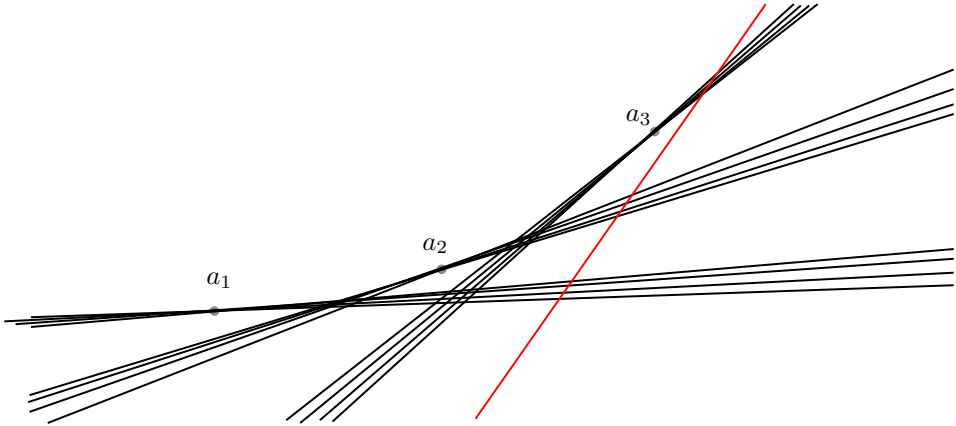
\begin{figure}[b]

\tikzset{every picture/.style={line width=0.75pt}} 

\begin{tikzpicture}[x=0.75pt,y=0.75pt,yscale=-0.7,xscale=0.7]

\draw    (1,232) -- (659,176) ;
\draw    (1,225) -- (659,202) ;
\draw    (-10,230) -- (659,183) ;
\draw    (-18,228) -- (659,193) ;
\draw    (13,301) -- (659,47) ;
\draw    (0,281) -- (659,79) ;
\draw    (0,293) -- (659,61) ;
\draw    (-1,286) -- (659,72) ;
\draw    (216,300) -- (545,0) ;
\draw    (183,299) -- (562,0) ;
\draw    (207,299) -- (550,1) ;
\draw    (193,301) -- (556,1) ;
\draw [color={rgb, 255:red, 252; green, 6; blue, 6 }  ,draw opacity=1 ]   (525,0) -- (318,298) ;
\draw  [draw opacity=0][fill={rgb, 255:red, 0; green, 0; blue, 0 }  ,fill opacity=0.47 ] (135.29,220.47) .. controls (135.23,218.58) and (133.65,217.08) .. (131.76,217.14) .. controls (129.86,217.19) and (128.37,218.77) .. (128.42,220.66) .. controls (128.47,222.56) and (130.05,224.05) .. (131.95,224) .. controls (133.84,223.95) and (135.34,222.37) .. (135.29,220.47) -- cycle ;
\draw  [draw opacity=0][fill={rgb, 255:red, 0; green, 0; blue, 0 }  ,fill opacity=0.47 ] (297.29,190.47) .. controls (297.23,188.58) and (295.65,187.08) .. (293.76,187.14) .. controls (291.86,187.19) and (290.37,188.77) .. (290.42,190.66) .. controls (290.47,192.56) and (292.05,194.05) .. (293.95,194) .. controls (295.84,193.95) and (297.34,192.37) .. (297.29,190.47) -- cycle ;
\draw  [draw opacity=0][fill={rgb, 255:red, 0; green, 0; blue, 0 }  ,fill opacity=0.47 ] (449.29,91.47) .. controls (449.23,89.58) and (447.65,88.08) .. (445.76,88.14) .. controls (443.86,88.19) and (442.37,89.77) .. (442.42,91.66) .. controls (442.47,93.56) and (444.05,95.05) .. (445.95,95) .. controls (447.84,94.95) and (449.34,93.37) .. (449.29,91.47) -- cycle ;

\draw (124,191) node [anchor=north west][inner sep=0.75pt]   [align=left] {$\displaystyle a_{1}$};
\draw (278,166) node [anchor=north west][inner sep=0.75pt]   [align=left] {$\displaystyle a_{2}$};
\draw (423,74) node [anchor=north west][inner sep=0.75pt]   [align=left] {$\displaystyle a_{3}$};

\end{tikzpicture}
\caption{Construction for $\mathcal{F}_{7,2}^4$. This family contains no $8$-cup.}

\end{figure}

It is easy to check that it contains no $l$ concurrent lines, no $(p+1)$-cup, no $3$-cap and no $4$-cell unbounded to the right. Moreover,
$$
|\mathcal{F}_{p,2}^l| = 
\begin{cases}
 (l-1) \frac{p}{2} & (\mbox{if} \,\, p \,\, \mbox{is even}) \\
 (l-1) \frac{p-1}{2} + 1 & (\mbox{if} \,\, p \,\, \mbox{is odd})
\end{cases}.
$$
Hence, for all $p \geq 2$,
$$
|\mathcal{F}_{p,2}^l| \geq \frac{l-1}{2} p - \frac{l-3}{2}.
$$
We can construct $\mathcal{F}_{2,q}^l$ similarly.

Now suppose we obtain $\mathcal{F}_{p-1,q}^l$ and $\mathcal{F}_{p,q-1}^l$ satisfying the properties.
We construct $\mathcal{F}_{p,q}^l$ in much the same way as Lemma 4.1 in \cite{ESFL}.

Let $a_1$ and $a_2$ be positive slope lines such that $a_2$ has greater slope than $a_1$ and they intersect above the $x$-axis. Take a sufficiently small $\varepsilon > 0$ and set $\mathcal{F}_{p,q}^l \coloneqq \mathcal{F}_{p-1,q}^l (a_1, \varepsilon) \cup \mathcal{F}_{p,q-1}^l (a_2, \varepsilon)$.
Now assume that $\mathcal{G} \subset \mathcal{F}_{p,q}^l$ defines a $4$-cell $P$ unbounded to the right.
By the definition of $\mathcal{F}_{p,q}^l$, $|\mathcal{G} \cap \mathcal{F}_{p-1,q}^l| \geq 1$ and $|\mathcal{G} \cap \mathcal{F}_{p,q-1}^l| \geq 1$.
Therefore, $P$ can not be a $4$-cell, a contradiction.
Now, suppose that $\mathcal{G} \subset \mathcal{F}_{p,q}^l$ defines a $(p+1)$-cup.
Since $a_2$ has greater slope than $a_1$, we can assume that $|\mathcal{G} \cap \mathcal{F}_{p,q-1}^l (a_2, \varepsilon)| \leq 1$. This implies $|\mathcal{G} \cap \mathcal{F}_{p-1,q}^l (a_1, \varepsilon)| \geq p$.  Therefore, there is a $p$-cup or a $(p+1)$-cup in $\mathcal{F}_{p-1,q}^l (a_1, \varepsilon)$, a contradiction.
Similarly, we can show that $\mathcal{F}_{p,q}^l$ contains no $(q+1)$-caps. 
Finally,
$$
\begin{aligned}
&|\mathcal{F}_{p,q}^l| = |\mathcal{F}_{p-1,q}^l| + |\mathcal{F}_{p,q-1}^l| \\
&\geq \frac{l-1}{2} \binom{p+q-3}{q-1} - \frac{l-3}{2} \binom{p+q-5}{q-2} + \frac{l-1}{2} \binom{p+q-3}{q-2} - \frac{l-3}{2} \binom{p+q-5}{q-3} \\
&= \frac{l-1}{2} \binom{p+q-2}{q-1} - \frac{l-3}{2} \binom{p+q-4}{q-2}.
\end{aligned}
$$
\end{proof}

\begin{prop}
Let $k \geq 2$.
If $n = 2k + 2$,
$$
\ESl (l,n) > \frac{l-1}{2} {\binom{2k-2}{k-1}}^2 - \frac{l-3}{2} \binom{2k-4}{k-2}  \binom{2k-2}{k-1}.
$$
If $n = 2k + 1$,
$$
\ESl (l,n) > \frac{l-1}{2} \bigg\{ \binom{2k-2}{k-1} + 1 \bigg\} \cdot \binom{2k-3}{k-1} - \frac{l-3}{2} \binom{2k-4}{k-2} \binom{2k-2}{k-1}.
$$

\end{prop}

\begin{proof} 
Suppose $n = 2k + 2$ for some $k \geq 2$. Let $\mathcal{F}_{k,k}$ be a family of $N \coloneqq \binom{2k-2}{k-1}$ lines containing no three concurrent lines, no $(k+1)$-cup, no $(k+1)$-cap and no $4$-cell unbounded to the right.
Reflecting $\mathcal{F}_{k,k}$ over $y$-axis, we get a family of lines $\mathcal{F}_{k,k}' \coloneqq \{ a_1, \ldots, a_N \}$ with no three concurrent lines, no $(k+1)$-cup, no $(k+1)$-cap and no $4$-cell unbounded to the left. Moreover, applying a suitable (unbounded-cell-preserving) affine transformation, we can assume every $a_i$ has a positive slope and they all intersect above $x$-axis.

Let $\mathcal{F}_{k,k}^l$ be a family of lines given in Proposition $3.1$.
Take a sufficiently small $\varepsilon > 0$ and set $\mathcal{F} = \bigcup_{i=1}^N \mathcal{F}_{k,k}^l (a_i, \varepsilon)$.
Then, 
$$
|\mathcal{F}| \geq \frac{l-1}{2} {\binom{2k-2}{k-1}}^2 - \frac{l-3}{2} \binom{2k-4}{k-2}  \binom{2k-2}{k-1}.
$$
To prove the inequality for $n = 2k + 2$,
it suffices to show that $\mathcal{F}$ contains no $l$ concurrent lines and no $n$ lines in convex position.
For all $i = 1, \ldots, N$, $\mathcal{F}_{k,k}^l (a_i, \varepsilon)$ contains no $l$ concurrent lines. Therefore, 
$\mathcal{F}$ does not contain an intersection point below the $x$-axis where $l$ lines of $\mathcal{F}$ intersect.
Moreover, if we take $\varepsilon > 0$ to be sufficiently small, we may also assume that 
$\mathcal{F}$ does not contain an intersection point above the $x$-axis where $l$ lines of $\mathcal{F}$ intersect.

Now, we show that for $n > 4$, $\mathcal{F}$ contains no $n$ lines in convex position.
The following argument is almost the same as the proof of Proposition $4.2$ in \cite{ESFL}.
Suppose $\mathcal{G} = \{ b_1, \ldots, b_n \} \subset \mathcal{F}$ defines an $n$-cell $C$. For each $i = 1, \ldots, N$, define $\mathcal{G}_i \coloneqq \mathcal{F}_{k,k}^l (a_i, \varepsilon) \cap \mathcal{G}$. 
Note that, for any subset $\mathcal{G'}$ of $\mathcal{G}$ defining a cell (write $C(\mathcal{G'})$), the set $C(\mathcal{G'})$ contains $C$. 
We need the following observation.
\begin{obs}
If $C(\mathcal{G}_i)$ is neither a cup nor a cap, then $|\mathcal{G} \backslash \mathcal{G}_i| \leq 1$.
\end{obs}

\begin{proof}
If $C(\mathcal{G}_i)$ is unbounded to the right, then $|\mathcal{G} \backslash \mathcal{G}_i| \leq 1$.
If $C(\mathcal{G}_i)$ is bounded, then $C(\mathcal{G}_j) = \emptyset$ for all $j \neq i$. Hence, $\mathcal{G}_i = \mathcal{G}$.
Additionally, if $C(\mathcal{G}_i)$ is unbounded to the left, then $\mathcal{G}_j = \emptyset$ for all $j \neq i$. Therefore, $\mathcal{G}_i = \mathcal{G}$.
\end{proof}
To complete the proof, consider the subset $I \coloneqq \{ i : |\mathcal{G}_i| \geq 2 \}$ of $\{1, \ldots, N\}$ and dividing into some cases depending on the number of $|I|$. From this point onward, exactly the same way as \cite{ESFL} (from p.6, l.35 to p.7, l.16) completes the proof.

If $I = \emptyset$, then $\mathcal{G}$ is essentially the same as a subset of $\mathcal{F}_{k,k}'$. Hence, $|\mathcal{G}| \leq 2k < n$, a contradiction.

\begin{figure}[htbp]

\tikzset{every picture/.style={line width=0.75pt}} 

\begin{tikzpicture}[x=0.52pt,y=0.5pt,yscale=-1,xscale=1]

\draw    (387.81,356) -- (620.04,0) ;
\draw    (396.86,355) -- (592.14,0) ;
\draw    (377.26,357) -- (635.87,-1) ;
\draw    (371.22,355) -- (657.74,-2) ;
\draw    (565,1) -- (511.78,354) ;
\draw    (552.93,-1) -- (519.76,355) ;
\draw    (574.05,0) -- (506.27,353) ;
\draw    (583.85,0) -- (498.76,355) ;
\draw [color={rgb, 255:red, 0; green, 0; blue, 0 }  ,draw opacity=1 ] [dash pattern={on 0.84pt off 2.51pt}]  (329,273) -- (660,105) ;
\draw  [draw opacity=0][fill={rgb, 255:red, 128; green, 128; blue, 128 }  ,fill opacity=0.23 ] (534.31,208) -- (532.81,213) -- (498.76,355) -- (396.86,355) -- (442.17,273) -- (462.2,242) -- (511.28,180) -- (537.31,167) -- (534.31,201) -- cycle ;
\draw    (1,196) -- (300,0) ;
\draw    (161,0) -- (17,356) ;
\draw    (300,253) -- (0,333) ;
\draw [color={rgb, 255:red, 208; green, 2; blue, 27 }  ,draw opacity=1 ]   (0,204) -- (299,11) ;
\draw    (2,210) -- (298,20) ;
\draw    (1,218) -- (300,28) ;
\draw  [draw opacity=0][fill={rgb, 255:red, 155; green, 155; blue, 155 }  ,fill opacity=0.3 ] (299,11) -- (300,253) -- (31,324) -- (106,136) -- cycle ;

\draw (418.4,246.4) node [anchor=north west][inner sep=0.75pt]    {$\mathcal{G}_{i}$};
\draw (523.31,328.4) node [anchor=north west][inner sep=0.75pt]    {$\mathcal{G}_{j}$};
\draw (177,182.4) node [anchor=north west][inner sep=0.75pt]    {$C(\mathcal{G} ')$};
\draw (1,165.4) node [anchor=north west][inner sep=0.75pt]  [color={rgb, 255:red, 208; green, 2; blue, 27 }  ,opacity=1 ]  {$a_{i} '$};
\draw (36,134.4) node [anchor=north west][inner sep=0.75pt]    {$\mathcal{G}_{i}$};

\end{tikzpicture}

\caption{}

\end{figure}

\begin{figure}

\tikzset{every picture/.style={line width=0.75pt}} 

\begin{tikzpicture}[x=0.52pt,y=0.5pt,yscale=-1,xscale=1]

\draw    (0,283) -- (323.05,51) ;
\draw    (1.91,277) -- (324,59) ;
\draw    (0,290) -- (324,35) ;
\draw    (323,137) -- (0,163) ;
\draw    (289,1) -- (157,358) ;
\draw  [draw opacity=0][fill={rgb, 255:red, 0; green, 0; blue, 0 }  ,fill opacity=0.18 ] (236,143) -- (157,358) -- (0.54,356.69) -- (-1,272) -- (107,206) -- (203,147) -- cycle ;
\draw [color={rgb, 255:red, 208; green, 2; blue, 27 }  ,draw opacity=1 ]   (334.89,284) -- (657.94,52) ;
\draw    (336.8,278) -- (658.89,60) ;
\draw    (334.89,291) -- (658.89,36) ;
\draw    (416,359) -- (647,1) ;
\draw    (641,0) -- (501.89,357) ;
\draw  [draw opacity=0][fill={rgb, 255:red, 0; green, 0; blue, 0 }  ,fill opacity=0.18 ] (437,211) -- (336.8,278) -- (338,0) -- (641,0) -- (629,27) -- (613,54) -- (561,114) -- (443,207) -- cycle ;
\draw    (659,1) -- (351,355) ;
\draw [color={rgb, 255:red, 208; green, 2; blue, 27 }  ,draw opacity=1 ]   (321,74) -- (-1,272) ;
\draw    (382.8,315.52) -- (356.79,350.37) ;
\draw    (378.84,321.06) -- (354,349.26) ;
\draw    (379.04,323.34) -- (353.8,346.98) ;
\draw    (381.24,313.29) -- (357.19,354.93) ;
\draw    (445.06,313.44) -- (422.31,350.51) ;
\draw    (443.61,319.14) -- (421.43,349.47) ;
\draw    (444.02,321.39) -- (421.02,347.22) ;
\draw    (443.3,311.36) -- (423.13,355.01) ;
\draw    (518.29,312.29) -- (504.52,353.55) ;
\draw    (518.16,318.17) -- (503.43,352.74) ;
\draw    (519.07,320.27) -- (502.52,350.64) ;
\draw    (516.1,310.66) -- (506.33,357.75) ;

\draw (1.11,241.4) node [anchor=north west][inner sep=0.75pt]    {$\mathcal{G}_{i}$};
\draw (15.84,196.22) node [anchor=north west][inner sep=0.75pt]  [font=\scriptsize,rotate=-0.58]  {$\leq k-cap$};
\draw (90,279.4) node [anchor=north west][inner sep=0.75pt]    {$C(\mathcal{G} ')$};
\draw (336,242.4) node [anchor=north west][inner sep=0.75pt]    {$\mathcal{G}_{i}$};
\draw (350.74,202.22) node [anchor=north west][inner sep=0.75pt]  [font=\scriptsize,rotate=-0.58]  {$\leq k-cup$};
\draw (362,106.4) node [anchor=north west][inner sep=0.75pt]    {$C(\mathcal{G} ')$};
\draw (299,91.4) node [anchor=north west][inner sep=0.75pt]  [color={rgb, 255:red, 208; green, 2; blue, 27 }  ,opacity=1 ]  {$a_{i} '$};
\draw (625,90.4) node [anchor=north west][inner sep=0.75pt]  [color={rgb, 255:red, 208; green, 2; blue, 27 }  ,opacity=1 ]  {$a_{i} '$};

\end{tikzpicture}
\caption{}

\end{figure}

If $|I| = 1$, set $I = \{ i \}$, choose an arbitrary line $a_i' \in \mathcal{G}_i$ and set $\mathcal{G}' = (\mathcal{G} \backslash \mathcal{G}_i) \cup \{ a_i' \}$.
Since $\mathcal{G}_i \subset \mathcal{F}_{k,k}^l (a_i, \varepsilon)$, if $|\mathcal{G} \backslash \mathcal{G}_i | \leq 1$, then, $|\mathcal{G}| = |\mathcal{G} \backslash \mathcal{G}_i | + |\mathcal{G}_i| \leq 1 + 2k < n$. Hence we may assume that $C(\mathcal{G}_i)$ is a cup or a cap.
If $C(\mathcal{G}')$ is unbounded to the right, then $\mathcal{G}_i = \{ a_i' \}$, a contradiction. See Figure 5 left.
If $C(\mathcal{G}')$ is bounded, then $\mathcal{G}_i = \{ a_i' \}$ also holds for the same reason. 
Since $\mathcal{G}_i$ forms a cup or a cap and the family of lines $\mathcal{F}_{k,k}'$ contains no $4$-cell unbounded to the left, if $C(\mathcal{G}')$ is unbounded to the left, $|\mathcal{G}| \leq k + 2 < n$ (Figure 6 left). If $C(\mathcal{G}')$ is a cup or a cap, then, $|\mathcal{G}| \leq k + (k - 1) < n$ (Figure 6 right), a contradiction.

If $|I| \geq 2$, then $|\mathcal{G} \backslash \mathcal{G}_i| \geq 2$ for each $i \in I$. By Observation $3$, $C(\mathcal{G}_i)$ is a cup or a cap for each $i \in I$. We divide into two more cases. 
If $|I| = 2$, set $I = \{ i, j \}$ with $i < j$. In this case, $C(\mathcal{G})$ must be an $n$-cell unbounded to the left. 
Moreover, $\mathcal{G}_i$ must define a $k$-cap and $\mathcal{G}_j$ must define a $k$-cup. See Figure 5 right.
Since $\mathcal{F}_{k,k}' = \{ a_1, \ldots, a_N \}$ contains no $4$-cell unbounded to the left, there can be at most one other line in $\mathcal{G}$. Hence, $|\mathcal{G}| \leq 2k + 1 < n$.

If $|I| \geq 3$, by the pigeonhole principle, there exist $\alpha > \beta$ such that $C(\mathcal{G}_\alpha)$ and $C(\mathcal{G}_\beta)$ are either both cups or both caps.
If $C(\mathcal{G}_\alpha)$ and $C(\mathcal{G}_\beta)$ are both cups, then at least one slope in $\mathcal{G}_\alpha$ can not be the part of $C(\mathcal{G})$, a contradiction.
If $C(\mathcal{G}_\alpha)$ and $C(\mathcal{G}_\beta)$ are both caps, then at least one slope in $\mathcal{G}_\beta$ can not be the part of $C(\mathcal{G})$, a contradiction. See Figure 7.

\begin{figure}[htbp]

\tikzset{every picture/.style={line width=0.75pt}} 

\begin{tikzpicture}[x=0.52pt,y=0.5pt,yscale=-1,xscale=1]

\draw    (0,283) -- (323.05,51) ;
\draw    (119.12,356) -- (314.47,1) ;
\draw [color={rgb, 255:red, 74; green, 144; blue, 226 }  ,draw opacity=1 ]   (114.35,358) -- (324,3) ;
\draw    (1.91,277) -- (324,59) ;
\draw    (0,290) -- (324,35) ;
\draw    (124.84,358) -- (303.04,0) ;
\draw  [draw opacity=0][fill={rgb, 255:red, 0; green, 0; blue, 0 }  ,fill opacity=0.18 ] (314.47,1) -- (275.4,73) -- (104.82,208) -- (0,283) -- (0.95,2) -- cycle ;
\draw [color={rgb, 255:red, 74; green, 144; blue, 226 }  ,draw opacity=1 ]   (338.94,281.45) -- (660,18) ;
\draw    (414.26,356.63) -- (610,-1) ;
\draw [color={rgb, 255:red, 0; green, 0; blue, 0 }  ,draw opacity=1 ]   (409.5,358.58) -- (619,1) ;
\draw    (340.76,274.64) -- (660,29) ;
\draw    (339.01,289.41) -- (657.74,3) ;
\draw    (420.02,358.69) -- (596,1) ;
\draw  [draw opacity=0][fill={rgb, 255:red, 0; green, 0; blue, 0 }  ,fill opacity=0.18 ] (557,107) -- (660,29) -- (659,357) -- (420.02,358.69) -- (475,245) -- cycle ;

\draw (41.66,185.4) node [anchor=north west][inner sep=0.75pt]  [font=\scriptsize]  {$\leq k-cup$};
\draw (100.61,250.4) node [anchor=north west][inner sep=0.75pt]  [font=\scriptsize]  {$\leq k-cup$};
\draw (7.11,289.4) node [anchor=north west][inner sep=0.75pt]    {$\mathcal{G}_{\beta }$};
\draw (154.79,320.4) node [anchor=north west][inner sep=0.75pt]    {$\mathcal{G}_{\alpha }$};
\draw (325.84,210.22) node [anchor=north west][inner sep=0.75pt]  [font=\scriptsize,rotate=-0.58]  {$\leq k-cap$};
\draw (462.56,270.06) node [anchor=north west][inner sep=0.75pt]  [font=\scriptsize,rotate=-0.58]  {$\leq k-cap$};
\draw (346.1,288.88) node [anchor=north west][inner sep=0.75pt]  [rotate=-0.58]  {$\mathcal{G}_{\beta }$};
\draw (390.62,320.39) node [anchor=north west][inner sep=0.75pt]  [rotate=-0.58]  {$\mathcal{G}_{\alpha }$};

\end{tikzpicture}
\caption{}
\end{figure}

If $n = 2k + 1$, consider $\mathcal{F} = \mathcal{F}_{k,k}^l (a_1, \varepsilon) \cup \bigcup_{i=2}^N \mathcal{F}_{k-1,k}^l (a_i, \varepsilon)$. Then, the same argument shows that $\mathcal{F}$ satisfies the property. Moreover,
$$
\begin{aligned}
|\mathcal{F}| &\begin{multlined}[t][11.5cm]
\geq \frac{l-1}{2} \binom{2k-2}{k-1} - \frac{l-3}{2} \binom{2k-4}{k-2} \\
+ \bigg\{ \binom{2k-2}{k-1} - 1 \bigg\} \cdot \bigg\{ \frac{l-1}{2} \binom{2k-3}{k-1} - \frac{l-3}{2} \binom{2k-4}{k-2} \bigg\} 
\end{multlined} \\
&= \frac{l-1}{2} \bigg\{ \binom{2k-2}{k-1} + 1 \bigg\} \cdot \binom{2k-3}{k-1} - \frac{l-3}{2} \binom{2k-4}{k-2} \binom{2k-2}{k-1}.
\end{aligned}
$$
\end{proof}

\begin{figure}[htbp]
\tikzset{every picture/.style={line width=0.75pt}} 

\begin{tikzpicture}[x=0.75pt,y=0.75pt,yscale=-0.75,xscale=0.72]

\draw    (191,2) -- (72,300) ;
\draw    (260,2) -- (0,293) ;
\draw    (0,233) -- (308,24) ;
\draw [color={rgb, 255:red, 208; green, 2; blue, 27 }  ,draw opacity=1 ]   (543,1) -- (424,299) ;
\draw [color={rgb, 255:red, 208; green, 2; blue, 27 }  ,draw opacity=1 ]   (612,1) -- (350,293.5) ;
\draw [color={rgb, 255:red, 208; green, 2; blue, 27 }  ,draw opacity=1 ]   (350,232.5) -- (659,26) ;
\draw [color={rgb, 255:red, 208; green, 2; blue, 27 }  ,draw opacity=1 ]   (523,1) -- (404,299) ;
\draw [color={rgb, 255:red, 208; green, 2; blue, 27 }  ,draw opacity=1 ]   (592,1) -- (349,276.5) ;
\draw [color={rgb, 255:red, 208; green, 2; blue, 27 }  ,draw opacity=1 ]   (349,220.5) -- (659,11) ;
\draw [color={rgb, 255:red, 208; green, 2; blue, 27 }  ,draw opacity=1 ]   (558,2) -- (439,300) ;
\draw [color={rgb, 255:red, 208; green, 2; blue, 27 }  ,draw opacity=1 ]   (627,2) -- (361,298.5) ;
\draw [color={rgb, 255:red, 208; green, 2; blue, 27 }  ,draw opacity=1 ]   (350,246.5) -- (659,41) ;
\draw  [draw opacity=0][fill={rgb, 255:red, 128; green, 128; blue, 128 }  ,fill opacity=1 ] (300,156.5) -- (323.4,156.5) -- (323.4,154) -- (339,159) -- (323.4,164) -- (323.4,161.5) -- (300,161.5) -- cycle ;
\draw  [draw opacity=0][fill={rgb, 255:red, 155; green, 155; blue, 155 }  ,fill opacity=0.3 ] (191,2) -- (135,141) -- (0,233) -- (2,2) -- cycle ;
\draw  [draw opacity=0][fill={rgb, 255:red, 155; green, 155; blue, 155 }  ,fill opacity=0.3 ] (514,110.25) -- (486,141) -- (350,232.5) -- (350,0) -- (558,2) -- cycle ;

\end{tikzpicture}
\caption{Some cups change into cups with larger size.}

\end{figure}
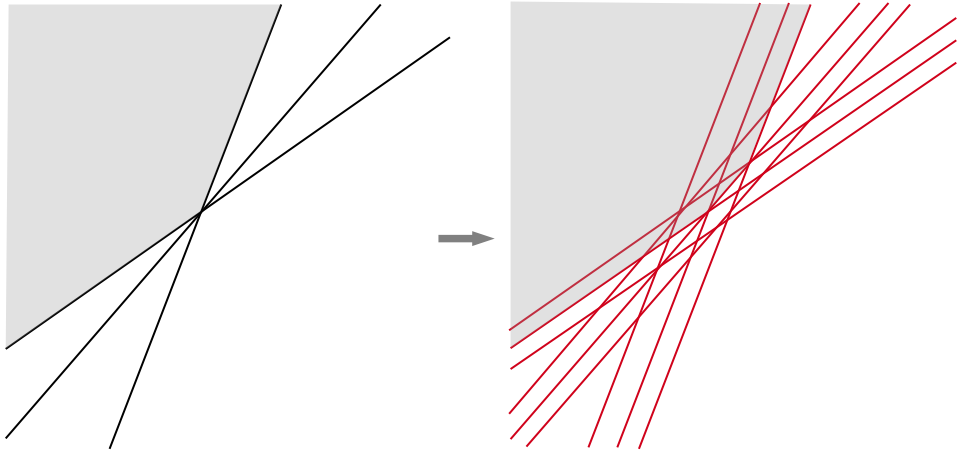

\,\\
\textbf{Remark.}
If we replace \(\mathcal{F}_{k,k}\) with \(\mathcal{F}_{k,k}^l\) in the above construction, a larger cup or cap will be formed, as shown in Figure 8.
\\

Before proving Theorem $1.2$, we note that by reflecting $\mathcal{F}_{k,k}^l$ (with no $l$ concurrent lines, no $(k+1)$-cup, no $(k+1)$-cap and no $4$-cell unbounded to the right) over a vertical line, we obtain a family of lines $\mathcal{F'}_{k,k}^l$ with no $l$ concurrent lines, no $(k+1)$-cup, no $(k+1)$-cap and no $4$-cell unbounded to the left.

\begin{proof}[Proof of Theorem \ref{thm1.2}]

The construction is almost the same as the one from Theorem 4.3 in \cite{ESFL}, that is realized by the union of two copies of $\mathcal{F}$ which are constructed in the previous proposition with a proper transformation. The size of the resulting family is $2|\mathcal{F}|$ which gives the best lower bound of $\ESl (l,n)$.
To be more precise, we construct as follows:

For $n = 2k + 2$, let $a$ be a line with slope $1$ and $b$ be a line with slope $-1$ and they intersect above the $x$-axis.
Construct the families $\mathcal{F}_{k,k}^3 (a, \varepsilon)$ and $\mathcal{F'}_{k,k}^3 (b, \varepsilon)$.
Take a sufficiently small $\varepsilon > 0$, then we can assume that both families intersect above the $x$-axis.
Reflect both families over the $x$-axis and move enough in the direction of the $y$-axis so that all intersections are above the $x$-axis.

Let $\mathcal{G} = \{ a_1, \ldots, a_{2N} \}$ be the resulting family where $N = \binom{2k-2}{k-1}$ and suppose the slope of $a_j$ is larger than the slope of $a_i$ for each $i < j$.
Finally, consider the family $\bigcup_{i=1}^N \mathcal{F'}_{k,k}^l (a_i, \varepsilon) \cup \bigcup_{i=N+1}^{2N} \mathcal{F}_{k,k}^l (a_i, \varepsilon)$.
See Figure 9.

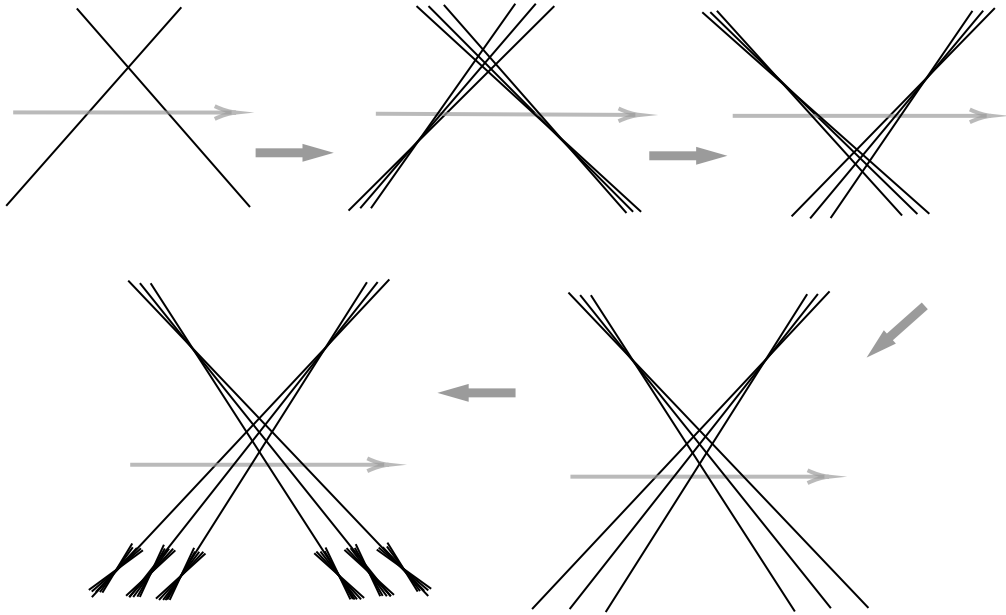
\begin{figure}[htbp]
\tikzset{every picture/.style={line width=0.75pt}} 

\begin{tikzpicture}[x=0.75pt,y=0.75pt,yscale=-0.75,xscale=0.75]

\draw    (48.01,2.19) -- (163.28,135.36) ;
\draw    (117.56,1.41) -- (1,134.58) ;
\draw [color={rgb, 255:red, 155; green, 155; blue, 155 }  ,draw opacity=0.69 ][line width=1.5]    (5.64,71.91) -- (151,72) ;
\draw [shift={(154,72)}, rotate = 180.03] [color={rgb, 255:red, 155; green, 155; blue, 155 }  ,draw opacity=0.69 ][line width=1.5]    (14.21,-4.28) .. controls (9.04,-1.82) and (4.3,-0.39) .. (0,0) .. controls (4.3,0.39) and (9.04,1.82) .. (14.21,4.28)   ;
\draw    (282.06,0.6) -- (418.24,138.57) ;
\draw    (365.8,0.6) -- (228.84,137.77) ;
\draw [color={rgb, 255:red, 155; green, 155; blue, 155 }  ,draw opacity=0.69 ][line width=1.5]    (247.01,72.8) -- (419.75,73.59) ;
\draw [shift={(422.75,73.6)}, rotate = 180.26] [color={rgb, 255:red, 155; green, 155; blue, 155 }  ,draw opacity=0.69 ][line width=1.5]    (14.21,-4.28) .. controls (9.04,-1.82) and (4.3,-0.39) .. (0,0) .. controls (4.3,0.39) and (9.04,1.82) .. (14.21,4.28)   ;
\draw    (353.47,-1) -- (236.63,136.17) ;
\draw    (339.83,-1) -- (243.77,136.17) ;
\draw    (274.1,0.6) -- (423.57,138.57) ;
\draw    (292.28,-0.2) -- (413.83,139.37) ;
\draw    (607.45,140.07) -- (469.61,4.64) ;
\draw    (523.72,141.62) -- (659,1.93) ;
\draw [color={rgb, 255:red, 155; green, 155; blue, 155 }  ,draw opacity=0.69 ][line width=1.5]    (484.53,74.21) -- (653.55,74.21) ;
\draw [shift={(656.55,74.21)}, rotate = 180] [color={rgb, 255:red, 155; green, 155; blue, 155 }  ,draw opacity=0.69 ][line width=1.5]    (14.21,-4.28) .. controls (9.04,-1.82) and (4.3,-0.39) .. (0,0) .. controls (4.3,0.39) and (9.04,1.82) .. (14.21,4.28)   ;
\draw    (536.07,143) -- (651.23,3.68) ;
\draw    (549.7,142.75) -- (644.09,3.81) ;
\draw    (615.4,139.92) -- (464.28,4.74) ;
\draw    (597.24,141.06) -- (474.01,3.76) ;
\draw    (58,397) -- (256,183.93) ;
\draw [color={rgb, 255:red, 155; green, 155; blue, 155 }  ,draw opacity=0.69 ][line width=1.5]    (83.53,308.21) -- (162,308.21) -- (252.55,308.21) ;
\draw [shift={(255.55,308.21)}, rotate = 180] [color={rgb, 255:red, 155; green, 155; blue, 155 }  ,draw opacity=0.69 ][line width=1.5]    (14.21,-4.28) .. controls (9.04,-1.82) and (4.3,-0.39) .. (0,0) .. controls (4.3,0.39) and (9.04,1.82) .. (14.21,4.28)   ;
\draw    (83,397) -- (248.23,185.68) ;
\draw    (107,399) -- (241.09,185.81) ;
\draw    (65,394) -- (84.9,361) ;
\draw    (56,392.17) -- (92,365.59) ;
\draw    (60.63,393.55) -- (88.01,364.02) ;
\draw    (57.93,393.27) -- (90.71,364.3) ;
\draw    (63.39,393.63) -- (85.25,363.57) ;
\draw    (89.94,396.89) -- (105.98,361.85) ;
\draw    (80.79,396.09) -- (113.55,365.61) ;
\draw    (85.54,396.93) -- (109.41,364.5) ;
\draw    (82.83,396.96) -- (112.13,364.47) ;
\draw    (88.3,396.71) -- (106.62,364.36) ;
\draw    (109.81,398.87) -- (126.08,363.93) ;
\draw    (100.67,398) -- (133.63,367.74) ;
\draw    (105.42,398.88) -- (129.49,366.6) ;
\draw    (102.7,398.89) -- (132.21,366.59) ;
\draw    (108.17,398.67) -- (126.7,366.45) ;
\draw    (281.88,396.23) -- (82.21,184.72) ;
\draw    (256.88,396.42) -- (89.99,186.41) ;
\draw    (232.9,398.61) -- (97.13,186.49) ;
\draw    (274.86,393.28) -- (254.7,360.44) ;
\draw    (283.84,391.38) -- (247.63,365.09) ;
\draw    (279.22,392.79) -- (251.61,363.48) ;
\draw    (281.92,392.49) -- (248.91,363.79) ;
\draw    (276.46,392.9) -- (254.37,363.01) ;
\draw    (249.94,396.37) -- (233.63,361.46) ;
\draw    (259.09,395.49) -- (226.08,365.28) ;
\draw    (254.34,396.38) -- (230.22,364.13) ;
\draw    (257.05,396.38) -- (227.5,364.13) ;
\draw    (251.58,396.17) -- (233.01,363.97) ;
\draw    (230.09,398.5) -- (213.54,363.69) ;
\draw    (239.22,397.56) -- (206.02,367.56) ;
\draw    (234.48,398.48) -- (210.15,366.39) ;
\draw    (237.19,398.47) -- (207.43,366.4) ;
\draw    (231.72,398.29) -- (212.94,366.21) ;
\draw    (351,405) -- (549,191.93) ;
\draw [color={rgb, 255:red, 155; green, 155; blue, 155 }  ,draw opacity=0.69 ][line width=1.5]    (376.53,316.21) -- (455,316.21) -- (545.55,316.21) ;
\draw [shift={(548.55,316.21)}, rotate = 180] [color={rgb, 255:red, 155; green, 155; blue, 155 }  ,draw opacity=0.69 ][line width=1.5]    (14.21,-4.28) .. controls (9.04,-1.82) and (4.3,-0.39) .. (0,0) .. controls (4.3,0.39) and (9.04,1.82) .. (14.21,4.28)   ;
\draw    (376,405) -- (541.23,193.68) ;
\draw    (400,407) -- (534.09,193.81) ;
\draw    (574.88,404.23) -- (375.21,192.72) ;
\draw    (549.88,404.42) -- (382.99,194.41) ;
\draw    (525.9,406.61) -- (390.13,194.49) ;
\draw  [draw opacity=0][fill={rgb, 255:red, 155; green, 155; blue, 155 }  ,fill opacity=1 ] (167,96) -- (198.2,96) -- (198.2,93) -- (219,99) -- (198.2,105) -- (198.2,102) -- (167,102) -- cycle ;
\draw  [draw opacity=0][fill={rgb, 255:red, 155; green, 155; blue, 155 }  ,fill opacity=1 ] (429,98) -- (460.2,98) -- (460.2,95) -- (481,101) -- (460.2,107) -- (460.2,104) -- (429,104) -- cycle ;
\draw  [draw opacity=0][fill={rgb, 255:red, 155; green, 155; blue, 155 }  ,fill opacity=1 ] (614.37,203.89) -- (591.13,224.7) -- (593.13,226.94) -- (573.63,236.35) -- (585.12,218) -- (587.13,220.23) -- (610.37,199.42) -- cycle ;
\draw  [draw opacity=0][fill={rgb, 255:red, 155; green, 155; blue, 155 }  ,fill opacity=1 ] (340,257) -- (308.8,257) -- (308.8,254) -- (288,260) -- (308.8,266) -- (308.8,263) -- (340,263) -- cycle ;
\end{tikzpicture}
\caption{Construction for  $\bigcup_{i=1}^N \mathcal{F'}_{k,k}^l (a_i, \varepsilon) \cup \bigcup_{i=N+1}^{2N} \mathcal{F}_{k,k}^l (a_i, \varepsilon)$.}
\end{figure}

For $n = 2k + 1$, consider the family 
$$\mathcal{F'}_{k,k}^l (a_1, \varepsilon) \cup \bigcup_{i=2}^N \mathcal{F'}_{k-1,k}^l (a_i, \varepsilon) \cup \mathcal{F}_{k,k}^l (a_{N+1}, \varepsilon) \cup \bigcup_{i=N+2}^{2N} \mathcal{F}_{k-1,k}^l (a_i, \varepsilon).
$$
By the same argument as in the proof of the previous proposition, we can show that these families contain no $n$ lines in convex position. 
\end{proof}

The lower bound in Theorem $1.2$ is not optimal. For example, Figure 10 shows that $\ESl (l, 5) > 2l$ while Theorem $1.2$ gives $\ESl (l, 5) > l + 3$.
It seems that determining the exact value of $\ESl (l, n)$ is a challenging problem.

\begin{figure}[h]
\tikzset{every picture/.style={line width=0.75pt}} 

\begin{tikzpicture}[x=0.75pt,y=0.75pt,yscale=-0.75,xscale=0.75]

\draw    (102,267.01) -- (458,0.24) ;
\draw    (77,267.01) -- (504,0.24) ;
\draw    (125,268.01) -- (416,1.24) ;
\draw    (115,267.01) -- (435,0.24) ;
\draw    (480,0.24) -- (85,269) ;
\draw    (385,1) -- (303,269) ;
\draw    (557.79,267.4) -- (202.21,0.07) ;
\draw    (582.79,267.44) -- (156.21,0) ;
\draw    (534.79,268.36) -- (244.21,1.13) ;
\draw    (544.79,267.38) -- (225.21,0.11) ;
\draw    (180.21,0.04) -- (574.78,269.41) ;
\draw    (276,1) -- (358,268) ;

\end{tikzpicture}
\caption{$\ESl (l, 5) > 2l$. The union of two $(l-1)$ concurrent lines and two lines.}
\end{figure}
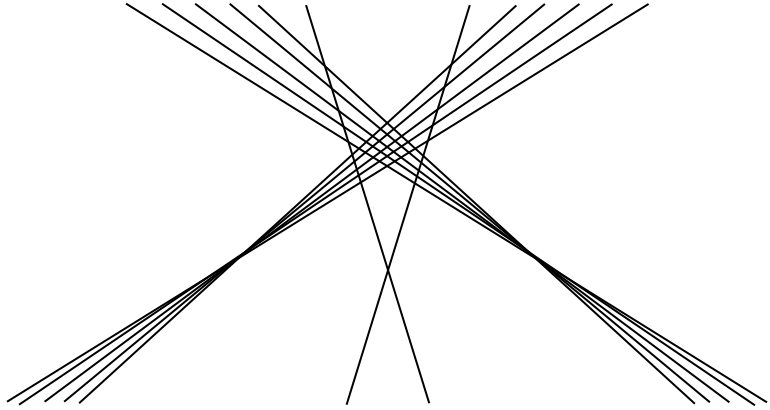


\section*{acknowledgments}
I\ would like to thank to my mentor, Masaki Tsukamto for many advices and care throughout preparation of this paper. Additionally, I am deeply grateful to Rokuyosha, a café in Kyoto, for their significant support in completing this paper.

\,\\
\centerline{Department of Mathematics, Kyoto University, Kyoto 606-8501, Japan}
\centerline{\href{furukawa.koki.38s@st.kyoto-u.ac.jp} {\nolinkurl{furukawa.koki.38s@st.kyoto-u.ac.jp}}}

\end{document}